\documentclass[12pt]{article}
\usepackage{amsmath, amsthm}

\theoremstyle{theorem}

\newtheorem{thm}{Theorem}[section]

\newtheorem{lem}[thm]{Lemma}
\newtheorem{prop}[thm]{Proposition}

\theoremstyle{definition}

\theoremstyle{remark}

\numberwithin{equation}{section}

\begin{document}

\title{A HILBERT SPACE ON LEFT-DEFINITE STURM-LIOUVILLE DIFFERENCE EQUATIONS}

\author{Rami AlAhmad\\
Dept. of Mathematics, Yarmouk University\\ Irbid, JORDAN -- 21163\\
e-mail: rami78uab@gmail.com}

\maketitle

\begin{abstract}

We investigate the discrete Sturm-Liouville problems
$$-\Delta(p\Delta u)(n-1)+q(n)u(n) = \l w(n) u(n),$$
where $p$ is strictly positive, $q$ is nonnegative and $w$ may change sign. If $w$ is positive, the $\ell^2$-space weighted by $w$ is a Hilbert space and it is customary to use that space for setting the problem. In the present situation the right-hand-side of the equation does not give rise to a positive-definite quadratic
form and we use instead the left-hand-side to definite such a form. We prove in this paper that this form determines a
Hilbert space (such problems are called left-definite).



\end{abstract}

\section{Introduction} \label{S1}

Let $\textbf{N}$ be the set of natural number. Define $
S(\textbf{N})$ to be the set of all the sequences over $\textbf{N}$
which are complex valued. If $u\in S(\textbf{N})$ then define
$\triangle:S(\textbf{N})\longrightarrow S(\textbf{N})$ to be the
first forward difference operator
 given by
\[(\triangle u)(n)=u(n+1)-u(n).\]
Using this definition,
\begin{equation}~\label{PR}
(\Delta (fg))(n)=g(n+1)(\Delta f)(n)+ f(n)(\Delta
g)(n).
\end{equation} Also, using the fact $\sum_{i=j}^k (\Delta
u)(i)=u(k+1)-u(j)$, we
get the summation by parts formula:
\begin{equation}~\label{sbp}
 \sum_{n=j}^{N}g(n+1)(\Delta
f)(n) =(fg)(N+1)-(fg)(j)-\sum_{n=j}^{N}f(n)(\Delta g)(n).
\end{equation}
This equation implies
$$
\sum\limits_{n=1}^{N}(p\Delta u)(n)\overline{\Delta v(n)}
$$ \vspace*{-10pt}
\begin{equation}\label{e.g.1}
=(p\Delta u)(N)\overline{v(N+1)}-(p\Delta
u)(0)\overline{v(1)}-\sum\limits_{n=1}^{N}\Delta(p\Delta
u)(n-1)\overline{v(n)}.
\end{equation}

We associate the term  {\em left-definite problem} with an inner
product associated with the left hand side of the equation $Lu=wf$ .

The left-definite spectral problem was first raised by Weyl in his
seminal paper \cite{weyl1} and treated by him in \cite{weyl2}. There
is now a large body of literature on the problem of determining
spectral properties for such systems. We mention here for instance
Niessen and Schneider \cite{Ni}, Krall \cite{Kr1,Kr2}, Marletta and
Zettl \cite{Zet}, Littlejohn and Wellman \cite{LW}.

In this paper, we are interested in studying an inner product
determined by the left-hand-side of the difference equation
\begin{equation}~\label{schro}
-(\Delta(p\Delta u))(n-1)+q(n)u(n)=\lambda w(n)u(n); \ \ n\geq 2,
\end{equation}

 Some spectral properties were discussed in~\cite{rr} related to left-hand-side of the equation
 \begin{equation}~\label{schro2}
-(\Delta^2 u)(n-1)+q(n)u(n)=\lambda w(n)u(n); \ \ n\geq 2.
\end{equation}
 Unlike the continuous case, the equation \eqref{schro} can not be transformed to \eqref{schro2}.

Now, for the solutions $\phi$ and $\theta$ of the equation
(\ref{schro}), we define the \emph{Wronskian}, $W_{\phi,\theta}$, to
be
\[W_{\phi,\theta}(n)=p(n)(\phi(n)(\Delta \theta)(n)-(\Delta\phi)(n)\theta(n)).\]
\begin{prop}
$W_{\phi,\theta}(n)$ is constant for all $n\in \textbf{N}.$
\end{prop}
\begin{proof}
 Using the product
rule~(\ref{PR})
\begin{align*}
(\Delta W_{\phi,\theta})(n)&=\phi(n+1)(\Delta(p\Delta\theta))(n)
+(\Delta \phi)(n)(p\Delta\theta)(n)\\&-(\theta(n+1)(\Delta(p\Delta\phi))(n)+(p\Delta \phi)(n)(\Delta\theta)(n))
\end{align*} Using the fact that
$\phi,\theta$ are solutions for ~(\ref{schro}), then
\begin{align*}
&(\Delta W_{\phi,\theta})(n)\\&=\phi(n+1)((q-\lambda
w)\theta)(n+1)-\theta(n+1)((q-\lambda w)\phi)(n+1)=0.
\end{align*}
 Hence, the Wronskian is constant.
\end{proof}
Our main interest is studying the equation \eqref{schro} where $\lambda$ is a complex parameter and where $q$ and $w$ are sequences with $q$ is defined on $\textbf{N}_0$ and assumes non-negative real values but is not identically equal to zero, $w$ is defined on $\textbf{ N}$ and real-valued, and
$p$ is defined on $\textbf{ N}_0$ and assumes strictly positive real values

Consider the operator  on the left-hand side of \eqref{schro} by
$L$, {i.e.,}
$$(Lu)(n)=-(\Delta (pu\Delta))(n-1)+(qu)(n), \quad n\in\textbf{N}.$$
Note that $L$ operates from $\textbf{C}^{\textbf{ N}_0}$ to $\textbf{C}^{\textbf{ N}}$.

\section{Main Result}

Due to the fact that the sign of $w$ is indefinite it is not convenient to phrase the spectral and scattering theory in the usual setting of a weighted $\ell^2$-space, since it is not a Hilbert space. Instead the requirement that $q$ is non-negative but not identically equal to zero allows us to define an inner product associated with the left hand side of the equation $Lu=wf$ giving rise to the term {\em left-definite problem}. To do so define the set
$$\mathcal{H}_1=\{u\in\textbf{ C}^{\textbf{N}_0}: \sum_{n=0}^\infty (p(n)|(\Delta u)(n)|^2+q(n)|u(n)|^2)<\infty\}$$
and introduce the scalar product
$$<u,v>=\sum_{n=0}^\infty (p(n)(\Delta u)(n)\overline{(\Delta v)(n)}+q(n)u(n)\overline{v(n)}).$$
The associated norm is denoted by $\|\cdot\|$. We will also use the norm in $\ell^2(\textbf{N}_0)$ which we denote by $\|\cdot\|_2$.
 We claim $\mathcal{H}_{1}$ with this norm is a complete space. Such a result plays a role in studying the spectral properties of \eqref{schro}.

 We start with the following sequence of lemmas:

\begin{lem}
If $m\geq n$, then for $u\in S(\textbf{N})$
\begin{equation}~\label{1.1}
|u(m)|\leq |u(n)|+ (\sum_{l=1}^{\infty}p(l)|(\Delta
u)(l)|^{2})^{1/2}(\sum_{l=n}^{m-1}\frac{1}{p(l)})^{1/2}.
\end{equation}
\end{lem}

\begin{proof}
\[|u(m)|-|u(n)|\leq |u(m)-u(n)|,\]
and
\[|u(m)-u(n)|=|\sum_{l=n}^{m-1}(\Delta u)(l)|\leq \sum_{l=n}^{m-1}|\Delta u(l)|.\]
Now, the inequality of Cauchy-Schwarz  gives that
\[\sum_{l=n}^{m-1}|\Delta u(l)|=
(\sum_{l=n}^{m-1}\sqrt{p}(l)|\Delta u|(l)(\frac{1}{\sqrt{p}(l)})$$
\vspace*{-10pt} $$ \leq(\sum_{l=n}^{m-1}p(l)|(\Delta
u)(l)|^{2})^{1/2}(\sum_{l=n}^{m-1}\frac{1}{p(l)})^{1/2}.\] By
combining the previous inequalities, we get:
\[|u(m)|-|u(n)|\leq (\sum_{l=n}^{m-1}p(l)|(\Delta u)(l)|^{2})^{1/2}(\sum_{l=n}^{m-1}\frac{1}{p(l)})^{1/2},\]
this inequality implies the required result since
\[(\sum_{l=n}^{m-1}p(l)|\Delta u)(l)|^{2})^{1/2}\leq
(\sum_{l=1}^{\infty}p(l)|\Delta u(l)|^{2})^{1/2}.\]
\end{proof}

\begin{lem}~\label{2}
If $r$ satisfies $\sum_{n=1}^{r}q({n})>0$, then for $1\leq n\leq
m\leq r<\infty$ and $u\in S(\textbf{N})$
\[
  |u(m)|\sum_{n=1}^{r}q({n})\leq(\sum_{n=1}^{r}q({n}))^{1/2}(\sum_{n=1}^{r}q(n)|u(n)|^{2})^{1/2}
  \] \vspace*{-10pt} \[+
C_r(\sum_{l=1}^{\infty}p(l)|\Delta
u(l)|^{2})^{1/2}\sum_{n=1}^{r}q({n}),
\]
where $C_r=(\sum_{l=1}^{r}\frac{1}{p(l)})^{1/2}$.
\end{lem}

\begin{proof}
The equation \eqref{1.1} gives that \[|u(m)|\leq |u(n)|+ C_r(\sum_{l=1}^{\infty}p(l)|(\Delta
u)(l)|^{2})^{1/2}.\]
Multiplying by $q(n)$ and taking the sum from 1 to $r$
with respect to $n$ give
\begin{equation}\label{1.2}
|u(m)|\sum_{n=1}^{r}q({n})\leq \sum_{n=1}^{r}q({n})|u(n)|+C_r(
\sum_{l=1}^{\infty}p(l)|\Delta
u(l)|^{2})^{1/2}\sum_{n=1}^{r}q({n}),
\end{equation}

Now, the inequality of Cauchy-Schwarz gives
\[\sum_{n=1}^{r}q({n})|u(n)|\leq
(\sum_{n=1}^{r}q({n}))^{1/2}(\sum_{n=1}^{r}q(n)|u(n)|^{2})^{1/2}.\]
Then~(\ref{1.2}) becomes
\[
  |u(m)|\sum_{n=1}^{r}q({n})\leq(\sum_{n=1}^{r}q({n}))^{1/2}(\sum_{n=1}^{r}q(n)|u(n)|^{2})^{1/2}
  \] \vspace*{-10pt} \[+
C_r(\sum_{l=1}^{\infty}|p(l)\Delta
u(l)|^{2})^{1/2}\sum_{n=1}^{r}q({n}).
\]
\end{proof}

We are ready to prove the following lemma:

\begin{lem}\label{2.1}
For any $N\in \textbf{N}$,  there exists $C_{N}$ such that
\[|u(m)| \leq C_{N}\|u\|_{\mathcal{H}_{1}},\]
for any $m$ such that $1\leq m \leq N$ and any
$u\in\mathcal{H}_{1}$.
\end{lem}

\begin{proof}For any $N\in \textbf{N}$ there exists $r\geq N$ such
that $\sum_{n=1}^r q(n)>0$. Now, Lemma~\ref{2} implies
\[|u(m)|\sum_{n=1}^{r}q(n)\leq \|u\|_{\mathcal{H}_{1}}(\sum_{n=1}^{r}q(n)^{1/2}+
C_r\sum_{n=1}^{r}q(n)),\] or
\[|u(m)|\leq
C_{N}\|u\|_{\mathcal{H}_{1}},\] where
\[C_{N}=C_r+(\sum_{n=1}^{r}q(n))^{-1/2}.\]
\end{proof}

The following lemma gives some properties for the Cauchy sequences
in $\mathcal{H}_1$.

\begin{lem}\label{cauchy} Let $n\mapsto u_{n}(\cdot)$ be a Cauchy sequence in
$\mathcal{H}_{1}$, then
\begin{enumerate}
\item There exists $v(\cdot)$ such that $(\sqrt{p}\Delta
u_{n})(\cdot)\longrightarrow v(\cdot)$ in $l^2(\textbf{N})$ as
$n\longrightarrow\infty$.
\item $\sqrt{q(\cdot)}u_n(\cdot)\longrightarrow \sqrt{q(\cdot)}u(\cdot)$ in $l^2(\textbf{N})$,
where $u(k)=\lim_{n\longrightarrow\infty}u_{n}(k)$ in $\textbf{C}$
for all $k\in\textbf{N}$.
\end{enumerate}
\end{lem}

\begin{proof}
\begin{enumerate}
\item If $n\mapsto u_{n}(\cdot)$
is a Cauchy sequence in $\mathcal{H}_{1}$, then for $\varepsilon>0$, there
exists $n_{0}$ such that for all $m,n\geq n_{0}$
\begin{equation}\label{conv.}
\|u_{m}(\cdot)-u_{n}(\cdot)\|_{\mathcal{H}_{1}}<\varepsilon
\end{equation}
consequently,
\[\|(\sqrt{p}\Delta u_{m})(\cdot)-(\sqrt{p}\Delta u_{n})(\cdot)\|_{l^{2}(\textbf{N})}< \varepsilon,\]
this means by the completeness of $l^{2}(\textbf{N})$ that there
exists $v(\cdot)$ such that, as $n\longrightarrow \infty$,
\begin{equation}~\label{l^2}
(\sqrt{p}\Delta u_{n})(\cdot)\longrightarrow v(\cdot) \text { in }
l^{2}(\textbf{N}).
\end{equation}
Therefore,
\begin{equation}~\label{l2}
(\sqrt{p}\Delta u_{n})(k)\longrightarrow v(k) \text { in } \textbf{C}.
\end{equation}
\item
Lemma~\ref{2.1} gives $K_r$ such that if $k\leq r$
\[|u_{m}(k)-u_{n}(k)|\leq K_r\|u_{m}(\cdot)-u_{n}(\cdot)\|_{\mathcal{H}_{1}}< K_r\varepsilon,\]
this means that $n\mapsto u_{n}(k)$ is a Cauchy sequence in
$\textbf{C}$. The completeness of the complex numbers $\textbf{C}$
gives the existence of $u\in S(\textbf{N})$ such that, as
$n\longrightarrow \infty$,
\begin{equation}~\label{c}
u_{n}(k)\longrightarrow u(k) \text { in } \textbf{C}
\end{equation}
and hence
\begin{equation}~\label{C}
\sqrt{q(k)}u_{n}(k)\longrightarrow \sqrt{q(k)}u(k) \text{ in\textbf{C}} .
\end{equation}
Moreover, equation~(\ref{conv.}) gives
\[\|\sqrt{q(\cdot)}u_{m}(\cdot)-\sqrt{q(\cdot)}u_{n}(\cdot)\|_{l^{2}(\textbf{N})}< \varepsilon.\]
Again by the completeness of $l^{2}(\textbf{N})$ then there exists
$\nu(\cdot)$ such that as $n\longrightarrow \infty$,
$\sqrt{q(\cdot)}u_{n}(\cdot)\longrightarrow \nu(\cdot)$ in
$l^{2}(\textbf{N})$ this means $\sum_{k=1}^\infty
|\sqrt{q(k)}u_{n}(k)- \nu(k)|^2\longrightarrow 0$. Hence, for any $k$,
\[\sqrt{q(k)}u_{n}(k)\longrightarrow \nu(k)\text{ in } \textbf{C},\]
which implies by ~(\ref{C}) $\nu(k)=\sqrt{q(k)}u(k)$.
\end{enumerate}
\end{proof}

\begin{prop}
The space $\mathcal{H}_{1}$ is complete.
\end{prop}

\begin{proof}
First, assume $n\mapsto u_n(\cdot)$ is a Cauchy sequence. Then using
Lemma~\ref{cauchy} there exist $u\in S(\textbf{N}) $ such that $u_n$
converges to $u$ pointwise  and $v(\cdot)\in l^2(\textbf{N})$ such
that $(\Delta u_n)(\cdot)\longrightarrow v(\cdot)$. This proves that
$u(k)=u(1)+\sum_{j=1}^{k-1}v(j)\in\mathcal{H}_1$. Also, this lemma
implies $(\sqrt{p}\Delta u_n)(\cdot)\longrightarrow (\sqrt{p}\Delta u)(\cdot)$
and $(\sqrt{q}u_n)(\cdot)\longrightarrow (\sqrt{q}u)(\cdot)$ in
$l^2(\textbf{N})$.

Moreover, one can prove that $u_{n}(\cdot)\longrightarrow u(\cdot)$
in $\mathcal{H}_1$ as follows. Since
\[\|u_{n}(\cdot)-u(\cdot)\|_{\mathcal{H}_{1}}=\sum_{k=1}^{\infty}p(k)|(\Delta u_{n})(k)-(\Delta u)(k)|^{2}
+\sum_{k=1}^{\infty}q(k)|u_{n}(k)-u(k)|^{2}, \] then
\[\|u_{n}(\cdot)-u(\cdot)\|_{\mathcal{H}_{1}}=\sum_{k=1}^{\infty}|(\sqrt{p(k)}(\Delta u_{n})(k)-v(k))|^{2}
\] \vspace*{-10pt}
\[ +\sum_{k=1}^{\infty}|(\sqrt{q(k)}(u_{n}(k)-u(k)))^{2}|.\] Using
Lemma\ref{cauchy} and the last equation, we get
$\|u_{n}(\cdot)-u(\cdot)\|_{\mathcal{H}_{1}}\longrightarrow 0$,
which means $u_{n}(\cdot)\longrightarrow u(\cdot)$ in
$\mathcal{H}_{1}$, and hence the Cauchy sequence in
$\mathcal{H}_{1}$ is convergent, this means $\mathcal{H}_{1}$ is
complete.
\end{proof}

\end{document}